\documentclass[12pt,a4paper]{article}
\usepackage[top=1in, bottom=1in, left=1in, right=1in]{geometry}
\usepackage[utf8]{inputenc}
\usepackage{amsmath}
\usepackage{textcomp}
\usepackage{amssymb}
\usepackage{diagbox}
\usepackage{ytableau}

\usepackage{youngtab}
\usepackage{graphicx,epstopdf}
\usepackage{ mathrsfs }
\usepackage{graphics}
\usepackage{stmaryrd}
\usepackage{amsfonts}
\usepackage{hyperref}
\usepackage{adjustbox}
\usepackage{amsthm}
\usepackage[skip=2pt]{caption}
\usepackage[figurename=Fig.]{caption}
 \usepackage[usenames,dvipsnames]{pstricks}
 \usepackage{epsfig}
 \usepackage{pst-grad} 
 \usepackage{pst-plot} 
 \usepackage{algorithm}
 \usepackage[noblocks]{authblk}
\usepackage{wasysym}
\usepackage[utf8]{inputenc}

\newtheorem{thm}{Theorem}[section]
\newtheorem{lem}{Lemma}[section]
\theoremstyle{definition}

\newtheorem{exmp}{Example}[section]
\newtheorem{con}{Conjecture}[section]

\title{\textbf{Total, Equitable  and Neighborhood Sum Distinguishing Total Colorings of Some Classes of Circulant Graphs}}
\author{S. Prajnanaswaroopa, \ J. Geetha, \ K. Somasundaram}
\affil{Department of Mathematics, Amrita School of Engineering-Coimbatore\\ Amrita Vishwa Vidyapeetham, India.

\{s\_prajnanaswaroopa, j\_geetha, s\_ sundaram\}@cb.amrita.edu sntrm4@rediffmail.com} 

\date{}
\begin{document}
\maketitle

\noindent\textbf{Abstract:} In this paper, we have obtained  the total chromatic as well as equitable and neighborhood sum distinguishing total chromatic numbers  of some classes of the Circulant graphs. \\


\section{Introduction}

Let $G=(V(G),E(G))$ be a graph with the sets of vertices $V(G)$ and edges $E(G)$ respectively. All the graphs considered here are finite, simple connected and undirected graphs. Total coloring of a simple graph is a coloring of the vertices and the edges such that any adjacent vertices or edges or edges and their incident vertices do not receive the same color. The \textit{total chromatic number}  of a graph \textit{G}, denoted by $ \chi''(G) $, is the minimum number of colors that suffice in a total coloring. It is clear that $ \chi''(G) \geq \Delta (G)+1 $, where $\Delta(G)$ is the maximum degree of \textit{G}.  \\

The Total Coloring Conjecture (TCC) (Behzad \cite{BEH} and Vizing \cite{VGV} independently proposed) is a famous open problem in graph theory  that states that any simple graph $G$ with maximum degree $\Delta(G)$ can be totally colored with at most $\Delta(G)+2$ colors. The progress on this conjecture had been low in the initial years but several advancements are in progress \cite{GNS}. The graphs that can be totally colored in $\Delta(G)+1$ colors are said to be type I and the graphs with the total chromatic number $\Delta(G)+2$ is said to be type II. \\

Equitable total coloring is a total coloring of a graph such that any two  color classes of the total coloring differ by at most one element. The minimum colors required in such a coloring of the graph is called the equitable total chromatic number, denoted by $\chi''_{=}(G)$. Fu \cite{FU} conjectured  that  
for every graph $G$, $G$ has an equitable total $k$-coloring for each $k \geq \text{ max } \{ \chi''(G), \Delta(G) +2 \}$. Later Wang \cite{WANG} proposed a weaker conjecture that $\chi''_{=}(G) \leq \Delta(G) + 2 .$\\

A proper total $k$ coloring of $G$ is called the neighborhood sum distinguishing total coloring if $\sum_{c}(u) \neq \sum_{c} (v)$ for each edge $uv$, where $\sum_c(u)$ is the sum of the color of the  vertex $u$ and the colors of edges incident with $u$.
The minimum number of colors required is called the neighborhood sum distinguishing total chromatic number, denoted by $\chi''_{\Sigma}(G)$. Pilsniak and Wozniak \cite{PIL} conjectured that $\Delta(G)+2\le\chi''_{\Sigma}(G)\le\Delta(G)+3$ for any graph $G$.\\

In an assignment of colors to the edges of a graph, a rainbow matching is a set of independent edges such that each edge in the independent set has a different color. 
For example, we can color the edges of any even cycle $C_{{2n}}$ such that it has a rainbow perfect matching (an $n-$ coloring of the edges with each perfect matching receiving $n$ colors is a rainbow coloring). \\

A Cayley graph, Cay$(G, S)$, for a group $H$ with symmetric generating set $S\subset H$  is a simple graph with vertices as all the elements of the group; and edges between each two elements of the form $g$ and $gs$, where $g\in H, s\in S$. Note that a symmetric subset implies $s\in S \implies s^{-1}\in S$. The set $S$ is also supposed to not have the identity element of the group. A circulant graph is a Cayley graph on the group $H=\mathbb{Z}_n$, the cyclic group on $n$ vertices. A power of cycle $C_n^k$ is a circulant graph on $\mathbb{Z}_n$ with generating set $\{1,2,\ldots,k,n-k,\ldots,n-2,n-1\}$.  \\

In other words,  given a sequence of positive integers $1\leq d_1<d_2<...<d_l\leq  \lfloor \frac{n}{2} \rfloor$, the circulant graph
    $G=C_n(d_1,d_2,...,d_l)$ has vertex set $V=Z_n=\{0,1,2,...,n-1\}$, two vertices $x$ and $y$ being adjacent iff $x=(y\pm
    d_i) \mod n$ for some $i, 1\leq i\leq l$ and a graph is a power of cycle, denoted $C^k_n$, $n$ and $k$ are integers, $1\leq k<\lfloor\frac{n}{2}\rfloor$, if
    $V(C^k_n)=\{v_0,v_1,...,v_{n-1}\}$ and $E(C^k_n)=E^1 \cup E^2 \cup ...\cup E^k$, where
    $E^i=\{e_0^i,e_1^i,...,e^i_{n-1}\}$ and $e_j^i=(v_j,v_{(j+i) \mod \ n })$, $0 \leq j \leq n-1$ and $1\leq i\leq k$.

    Campos and de Mello \cite {CAM} 
    verified the TCC for power of cycle $C_n^k, n$ even and
    $2<k<\frac{n}{2}.$ Also, they showed that one  can obtain a
    $\Delta(C_n^k)+2$-total coloring for these graphs in polynomial time. They also proved that $C_n^k$ with $n\cong 0 \mod (\Delta(C_n^k)+1)$ are
    type-I and they proposed the following conjecture.  
    \begin{con}
    Let $G=C_n^k$, with $2\leq k< \lfloor \frac{n}{2}
        \rfloor$. Then,

        $\chi''(G)=\begin{cases}\Delta(G)+2, & \text{ if }k>\frac{n}{3}-1    \text{ and } n  \text{ is \ odd}\\ \Delta(G)+1,  &
        \text{ otherwise}.  
    \end{cases}$

    \end{con}
  Geetha et al. \cite{Gee} proved the Campos and de Mello’s conjecture for some classes of powers of cycles. Also, they verify the TCC for complement of powers of cycles.\\
  
 A Latin square is an $n\times n$ array consisting of $n$ entries of numbers (or symbols) with each row and column containing only one instance of each element. This means the rows and columns are permutations of one single $n$ vector with distinct entries. A Latin square is said to be commutative if it is symmetric.  A Latin square containing numbers is said to be idempotent if each diagonal element contains the number equal to its row (column) number. In addition, if the rows of the Latin square are just cyclic permutations (one-shift of the elements to the right) of the previous row, then the Latin square is said to be circulant (anti-circulant, if the cyclic permutations are actually left shifts), the matrix (corresponds to the Latin square) can be generated from a single row vector.  The Latin square

\begin{center}

\begin{tabular}{|c|c|c|c|c|c|c|}\hline1&k+2&2&k+3&\ldots&2k+1&k+1\\\hline k+2&2&k+3&3&\ldots&k+1&1\\\hline\ldots&\ldots&\ldots&\ldots&\ldots&\ldots&\ldots\\\hline\ldots&\ldots&\ldots&\ldots&\ldots&\ldots&\ldots\\\hline k+1&1&k+2&2&\ldots&k&2k+1\\
\hline
\end{tabular}

\end{center}

\noindent is    anti-circulant, commutative and idempotent.  
The entries of the square are as follows:
\[ L=(l_{ij})=\begin{cases} m,\ \text{ if }\ i+j=2m\\k+1+m, \  \text{ if } i+j=2m+1. \end{cases}\]

From the above, it can be easily seen that the Latin square corresponding to the matrix $L$ is commutative, idempotent and also anti-circulant.
The following lemma is due to Stong \cite{STO} (Corollary 2.3.1).
\begin{lem}
 If $G$ is abelian and has even order then Cay$(G, S)$ is 1-factorizable. 
\end{lem}
\section{Results on Powers of Cycles}
\begin{thm} Let $G=C_n^k$ be a power of cycle graphs and $(k+i)|n$ for some $1\le i\le k+1$. Then $G$ satisfies TCC. In particular $G$ is type I if $n$ is even.
\end{thm}
\begin{proof}

\noindent Case 1. $n$ is even.\\

Let  $n\equiv0 \bmod (k+i)$ with $1\le i\le k+1$ and let us take $k+i=2m+1$. We also know that there exists a commutative idempotent Latin square of odd order, $k+i$ in this case, which we call $C'$. Now, we consider two tableau of the form\\
Tableau $B'$:\\
\\
\begin{tabular}{|c|c|c|c|c|}\hline
$k+i+1$&&&&\\
\hline$k+i+2$&$k+i+1$&&&\\
\hline$k+i+3$&$k+i+2$&$k+i+1$&&\\
\hline\ldots&\ldots&\ldots&$k+i+1$&\\
\hline$2k+1$&\ldots&\ldots&\ldots&$k+i+1$\\
\hline
\end{tabular}\\
\\

Tableau $A'$:\\
\\
\begin{tabular}{|c|c|c|c|c|}\hline
$2k+1$&$2k$&$2k-1$&$\ldots$&$k+i+1$\\
\hline &$2k+1$&$2k$&$\ldots$&$k+i+2$\\
\hline &&$2k+1$&$2k$&$\ldots$\\
\hline & & & $2k+1$ & $2k$\\
\hline & & & & $2k+1$\\
\hline
\end{tabular}




 Now, arranging the two tableau and the idempotent and commutative Latin square of order $k+1$ in the below fashion, would give us the color matrix desired, with $2k+1$ colors. The portion $C$ is the portion of the Latin square $C'$ which fits in the color matrix. 

\begin{tabular}{|c|c|c|c|c|c|}\hline
$C$&\diagbox{$B$}{}&&&&\diagbox{}{$A$}\\\hline
\diagbox{}{$B^T$\\}&C&\diagbox{$A^T$\\}{}&&&\\\hline
&\diagbox{}{$A$}&C&\diagbox{$B$}{}&&\\\hline
&&\diagbox{}{$B^T$\\}&C&\diagbox{$A^T$\\}{}&\\\hline
&&&\diagbox{}{$A$}&C&\diagbox{$B$}{}\\\hline
\diagbox{$A^T$\\}{}&&&&\diagbox{}{$B^T$\\}&$C$\\\hline\end{tabular}\\
\\
In case $i=1$, the entries of $C'$ are written wholly, so that the tableau $A'$ and $B'$ are equal to the tableau $A$ and $B$. In case $i>1$, the tableau $A'$ and $B'$ could be modified to accommodate the missed numbers in the color matrix,  which are deleted from the commutative idempotent Latin square $C'$, that is, the portion of the $k+i$ Latin square starting from the $(k+2)^{nd}$ position in the first row, $(k+3)^{rd}$  position in the second row and so on, is cut and juxtaposed on the tableau $A'$ and $B'$ to give us the tableau $A$ and $B$. In particular, if $i=k+1$, the tableau $A'$ and $B'$ are wholly replaced with the portions deleted from the Latin square $C'$. The portion deleted from the Latin square, in case $i>1$ would be $D'$ and its transpose, where $D$ given by:\\
\\
\begin{tabular}{|c|c|c|c|c|}\hline$e_{1,k+2}$&$e_{1,k+3}$&$\ldots$&$\ldots$&$e_{1,k+i}$\\\hline&$e_{2,k+3}$&$e_{2,k+4}$&$\ldots$&$e_{2,k+i}$\\\hline&&$\ldots$&$\ldots$&$e_{3,k+i}$\\\hline&&&$\ldots$&$e_{4,k+i}$\\\hline&&&&$e_{k+i,k+i}$\\\hline\end{tabular},\\
\\
where $e_{ij}$ denote the entries of the Latin square $C'$. Note that the portion $A'$ and $B'$ are actually only edge coloring of the subgraph formed after removing the edges of $C_n^{\lfloor \frac{k+i-1}{2}\rfloor}$. The circulant graph with generating set  $\{\lfloor \frac{k+i-1}{2}\rfloor+1, \lfloor \frac{k+i-1}{2}\rfloor+2,\ldots, k, n-\lfloor \frac{k+i-1}{2}\rfloor-1, n-\lfloor \frac{k+i-1}{2}\rfloor-2, \ldots, n-k \}$ is connected and the generating set is same for its group. Hence by Lemma 1.1, the circulant graph is 1-factorable. Therefore the circulant graph is  $d$ edge colorable, where $d$ is the degree of the subgraph thereby ensuring the compatibility of the various tableau and hence the total coloring of the graph.

\noindent Case 2. $n$ is odd.\\
We first give a $(k+i)$-partial total coloring of the graph that corresponds to totally coloring the subgraph $C_n^{\left\lfloor\frac{k+i}{2}\right\rfloor}$ by using the similar concept of the above case. That is, we first totally color the graph induced by the first $k+i$ vertices as in a commutative idempotent Latin square of order $k+i$ and repeat the same pattern for the remaining vertices, taken $k+i$ at a time (as in above case) and then extend it to the full total coloring by coloring the remaining edges, which can be always done in $2k-(k+i)+2=k-i+2$ colors by using the Vizing's theorem. Hence, the total  colors used is $2k+2$, which thereby verifies TCC for such graphs. 
\end{proof}
\begin{exmp}
For example, we see that the graph $C_{21}^6$ can be totally colored by first partially giving the part of the graph $C_{21}^3$ a type I total coloring and then extending it to the whole graph by giving the remaining edges a class 2 edge coloring, which is guaranteed by Vizing's theorem.The first part of the color matrix can be seen in Table 1.
\begin{table}[h]
\begin{adjustbox}{width=\columnwidth, center}
\begin{tabular}{|l|l|l|l|l|l|l|l|l|l|l|l|l|l|l|l|l|l|l|l|l|l|}\hline
   & 0  & 1  & 2  & 3  & 4  & 5  & 6  & 7  & 8  & 9  & 10 & 11 & 12 & 13 & 14 & 15 & 16 & 17 & 18 & 19 & 20 \\\hline
0  & 1  & 5  & 2  & 6  &   &  &   &    &    &    &    &    &    &    &    &    &  &  & 3  & 7  & 4  \\\hline
1  & 5  & 2  & 6  & 3  & 7  &   &  &  &    &    &    &    &    &    &    &    &    &   &  & 4  & 1  \\\hline
2  & 2  & 6  & 3  & 7  & 4  & 1  &  &  &  &    &    &    &    &    &    &    &    &    &  &  & 5  \\\hline
3  & 6  & 3  & 7  & 4  & 1  & 5  & 2  &  &  &  &    &    &    &    &    &    &    &    &    &  &   \\\hline
4  &   & 7  & 4  & 1  & 5  & 2  & 6  & 3  &  &  &  &    &    &    &    &    &    &    &    &    & \\\hline
5  &  &   & 1  & 5  & 2  & 6  & 3  & 7  & 4  &   &  &  &    &    &    &    &    &    &    &    &    \\\hline
6  &    &  &   & 2  & 6  & 3  & 7  & 4  & 1  & 5  &   &  &  &    &    &    &    &    &    &    &    \\\hline
7  &    &    &  &  & 3  & 7  & 4  & 1  & 5  & 2  & 6  &  & &  &    &    &    &    &    &    &    \\\hline
8  &    &    &    &  &   & 4  & 1  & 5  & 2  & 6  & 3  & 7  &  &  &  &    &    &    &    &    &    \\\hline
9  &    &    &    &    &  &   & 5  & 2  & 6  & 3  & 7  & 4  & 1  &   &  &  &    &    &    &    &    \\\hline
 10&    &    &    &    &    &  &  & 6  & 3  & 7  & 4  & 1  & 5  & 2  &   &  &  &    &    &    &    \\\hline
11 &    &    &    &    &    &    &  &  & 7  & 4  & 1  & 5  & 2  & 6  & 3  &   &  &  &    &    &    \\\hline
12 &    &    &    &    &    &    &    &  &  & 1  & 5  & 2  & 6  & 3  & 7  & 4  &   &  &  &    &    \\\hline
13 &    &    &    &    &    &    &    &    &  &   & 2  & 6  & 3  & 7  & 4  & 1  & 5  &   &  &  &    \\\hline
14 &    &    &    &    &    &    &    &    &    &  &   & 3  & 7  & 4  & 1  & 5  & 2  & 6  &   &  &  \\\hline
15 &    &    &    &    &    &    &    &    &    &    &  &   & 4  & 1  & 5  & 2  & 6  & 3  & 7  &   &  \\\hline
16 &  &    &    &    &    &    &    &    &    &    &    &  &   & 5  & 2  & 6  & 3  & 7  & 4  & 1  &   \\\hline
17 &  &  &    &    &    &    &    &    &    &    &    &    &  &   & 6  & 3  & 7  & 4  & 1  & 5  & 2  \\\hline
18 & 3  &  &  &    &    &    &    &    &    &    &    &    &    &  &  & 7  & 4  & 1  & 5  & 2  & 6  \\\hline
19 & 7  & 4  &  &  &    &    &    &    &    &    &    &    &    &    &  &   & 1  & 5  & 2  & 6  & 3  \\\hline
20 & 4  & 1  & 5  &   &  &    &    &    &    &    &    &    &    &    &    &  &   & 2  & 6  & 3  & 7  \\\hline
\end{tabular}
\end{adjustbox}
\caption{Partial total Coloring of $C_{21}^6$.}
\label{table:1}
\end{table}
The color matrix is then extended by giving a suitable $7$ coloring of the edges of the remaining edges( which are the edges of the Cayley graph on $\mathbb{Z}_{21}$ with generating set as $\{4,5,6,15, 16, 17\}$), which is guaranteed by Vizing's theorem. Hence, the total colors required for this graph is $7+7=14$.
\end{exmp}
 There are several classes of graphs satisfy the condition of Theorem 2.1. For example, if $n$ were highly composite, we have several classes of such graphs. 
The canonical total coloring of a complete graph of order $n$ is coloring the vertices and edges of the complete graph in a particular fashion which we describe here. First, we observe that the complete graph can be written as the circulant graph with generating set consisting of all elements of the cyclic group of order $n$ except $0$, that is $\{1, 2, \ldots, n-1\}$. Now, the total color matrix is filled as follows:\\

\noindent (i). The vertices are colored  $1-2-\ldots-n$, corresponding to filling the diagonal in the pattern $1-2-\ldots-n$.\\

\noindent (ii). For the first row of the color matrix, the colors corresponding to the generator $s_i$ is $\frac{s_i}{2}+1$ if $s_i$ is even and $\left\lceil\frac{n}{2}\right\rceil+s_i$ if $s_i$ is odd.\\

\noindent (iii). The remaining rows of the total color matrix are filled by the circulant pattern, that is rotating the first row to the left by one color cyclically.\\

Note that the above canonical coloring implies that if we have color $x$ in the first row of the total color matrix at position $s_i$, we have the color $x-s_i\pmod n$ at position $n-s_i$ in the first row.\\
 \begin{thm}
The equitable total chromatic number of a power of cycle graph, $G=C_n^k$ with $(2k+1)|n$, $n$ is even, is equal to its total chromatic number. Further, $\chi''_{\Sigma}(G) \le 2k+3=\Delta(G)+3$.
 \end{thm}
 \begin{proof}
 Let $G=C_n^k$ and $n=2m(2k+1).$ From the case 1 of the previous Theorem, the graph $C_n^k$ is type I and it is easy to see that the equitable total chromatic of $C_n^k$ is equal to its total chromatic number (we remove the vertices in each color classes, the remaining graphs have a perfect matching).\\

Secondly, we see that the edges of the hamiltonian cycle formed by the generator $1$  in the totally colored graph have two rainbow perfect matchings. We then give two new colors to the two rainbow perfect matchings, which will then give us a different sum of the colors incident at each vertex for each adjacent pair of vertices (as the row/column sum for any consecutive $2k+1$ rows/ columns will be different in the total color matrix). Hence the number of colors used in such a coloring is $2k+1+2=2k+3$. However, it may be possible to further reduce the number of total colors required by $1$, a possibility which we cannot rule out. Hence, the neighbourhood sum distinguishing total chromatic number, $\chi''_{\Sigma} (G) \le2k+3$.
 \end{proof}
\begin{exmp}

Let us consider the total coloring of the graph $C_{18}^4$. In this case , we have $2(4)+1=9|18$ and the graph is type I. We could write its total color matrix as in Table 2.
\begin{table}[h]
\begin{adjustbox}{width=\textwidth}
\begin{tabular}{|c|c|c|c|c|c|c|c|c|c|c|c|c|c|c|c|c|c|c|}
\hline&0&1&2&3&4&5&6&7&8&9&10&11&12&13&14&15&16&17\\\hline0&1&6&2&7&3&&&&&&&&&&8&4&9&5\\\hline1&6&2&7&3&8&4&&&&&&&&&&9&5&1\\\hline2&2&7&3&8&4&9&5&&&&&&&&&&1&6\\\hline3&7&3&8&4&9&5&1&6&&&&&&&&&&2\\\hline4&3&8&4&9&5&1&6&2&7&&&&&&&&&\\\hline5&&4&9&5&1&6&2&7&3&8&&&&&&&&\\\hline6&&&5&1&6&2&7&3&8&4&9&&&&&&&\\\hline7&&&&6&2&7&3&8&4&9&5&1&&&&&&\\\hline8&&&&&7&3&8&4&9&5&1&6&2&&&&&\\\hline9&&&&&&8&4&9&5&1&6&2&7&3&&&&\\\hline10&&&&&&&9&5&1&6&2&7&3&8&4&&&\\\hline11&&&&&&&&1&6&2&7&3&8&4&9&5&&\\\hline12&&&&&&&&&2&7&3&8&4&9&5&1&6&\\\hline13&&&&&&&&&&3&8&4&9&5&1&6&2&7\\\hline14&8&&&&&&&&&&4&9&5&1&6&2&7&3\\\hline15&4&9&&&&&&&&&&5&1&6&2&7&3&8\\\hline16&9&5&1&&&&&&&&&&6&2&7&3&8&4\\\hline17&5&1&6&2&&&&&&&&&&7&3&8&4&9\\\hline
\end{tabular}

\end{adjustbox}
\caption{Equitable total coloring of $C_{18}^4$.}
\label{table:2}
\end{table}
\\
We modify the two rainbow perfect matchings formed by the generating elements $\{1,17\}$ with two new colors as explained in the theorem to give a coloring such that the neighbourhood sum distinguishing total chromatic number is $11$ from the total chromatic number $9$. Hence, the  Table 2 could be modified  to Table 3.
\begin{table}
\begin{adjustbox}{width=\textwidth}

\begin{tabular}{|c|c|c|c|c|c|c|c|c|c|c|c|c|c|c|c|c|c|c|}
\hline&0&1&2&3&4&5&6&7&8&9&10&11&12&13&14&15&16&17\\\hline0&1&10&2&7&3&&&&&&&&&&8&4&9&11\\\hline1&10&2&11&3&8&4&&&&&&&&&&9&5&1\\\hline2&2&11&3&10&4&9&5&&&&&&&&&&1&6\\\hline3&7&3&10&4&11&5&1&6&&&&&&&&&&2\\\hline4&3&8&4&11&5&10&6&2&7&&&&&&&&&\\\hline5&&4&9&5&10&6&11&7&3&8&&&&&&&&\\\hline6&&&5&1&6&11&7&10&8&4&9&&&&&&&\\\hline7&&&&6&2&7&10&8&11&9&5&1&&&&&&\\\hline8&&&&&7&3&8&11&9&10&1&6&2&&&&&\\\hline9&&&&&&8&4&9&10&1&11&2&7&3&&&&\\\hline10&&&&&&&9&5&1&11&2&10&3&8&4&&&\\\hline11&&&&&&&&1&6&2&10&3&11&4&9&5&&\\\hline12&&&&&&&&&2&7&3&11&4&10&5&1&6&\\\hline13&&&&&&&&&&3&8&4&10&5&11&6&2&7\\\hline14&8&&&&&&&&&&4&9&5&11&6&10&7&3\\\hline15&4&9&&&&&&&&&&5&1&6&10&7&11&8\\\hline16&9&5&1&&&&&&&&&&6&2&7&11&8&10\\\hline17&11&1&6&2&&&&&&&&&&7&3&8&10&9\\\hline
\end{tabular}

\end{adjustbox}
\caption{Neighborhood sum distinguishing total coloring of $C_{18}^4$.}
\label{table:3}
\end{table}
\end{exmp}
 \section{Results on Circulant Graphs}

In this section, we obtain upper bound of total coloring for certain classes of circulant graphs.

\begin{thm} Let $G$ be a circulant graph with even order $n$, $\Delta(G) \ge\frac{n}{2}$, and the generating set $S$ satisfying the following properties: \\
\noindent (i)The set does not have the involution element $\frac{n}{2}.$\\
(ii)The set has a subset $S_1$ of the form $\{1,2,3\ldots,\frac{n}{4},n-\frac{n}{4},\ldots n-1\}$ (which is a generating set of the power of cycle $C_n^\frac{n}{4}$).  \\
(iii) The complement of the subset $S_1$ in $S$ generates the whole group.\\
Then $G$ satisfies the TCC.
\end{thm}
\begin{proof}

We bifurcate the total color matrix into two parts, first one consisting of partial total coloring of the graph induced by the generating set $\{1,2,3,\ldots,\frac{n}{4},\ldots,n-1\}$. The second part consisting of edge coloring of the remaining edges. The bifurcation of the graph is done by using the generating set $S_1$ and which is a power of cycle. The second graph is induced by $\overline{S_1}$. \\

Here, the first part of the total color matrix requires only $\frac{n}{2}+2$ colors by using the result of Campos and de Mello \cite {CAM}. From Lemma 1.1, the graph induced by the remaining edges form a Cayley graph and it is 1-factoriziable.  Therefore we can color the remaining edges with $\Delta(G)-\frac{n}{2}$ colors. This gives us the full total coloring of the original circulant graph using $\Delta(G)-\frac{n}{2} + \frac{n}{2}+2=\Delta(G) + 2$ colors thus satisfying TCC.
\end{proof}
\begin{exmp}
Let us consider the total coloring of the cayley graph on $\mathbb{Z}_{20}$ with generating set $\{1,2,3,4,5,7,8,12,13,15,16,17,18,19\}$. Here, we have degree of the graph equal to $14$ . The first part consists of the graph corresponding to the generating subset $\{1,2,3,4,5,15,16,17,18,19\}$. We can follow the process given in $\cite{CAM}$ to give a total $12$ coloring of the subgraph induced by this subset. Since the complement of the subset is $\{7,8,12,13\}$ generates the whole group, therefore,  we could color the remaining edges with $4$ extra colors. This implies that we could give a total coloring of $G$ in $12+4=16$ colors which is exactly $\Delta(G)+2$ colors here, or, $G$ satisfies TCC.  

\end{exmp}

\begin{thm}
If a circulant graph $G$ of even order $n$ has the generating set $S=\{s_i\}$ with $|S|=\frac{n}{2}-2$ such that none of $s_i+s_j\neq \frac{n}{2}, \ s_i,s_j\in S$ and $\frac{n}{2}\notin S$, then  $\chi''(G)\le\frac{n}{2}+1=\Delta(G)+3.$
\end{thm}
\begin{proof}
We begin by noticing that we can partition the vertices  into $\frac{n}{2}$ independent sets as $\{0,\frac{n}{2}\}, \{1,\frac{n}{2}+1\},\ldots,\{\frac{n}{2}-1, n-1\}$  as $\frac{n}{2}\notin S$. We observe that the induced subgraphs formed by the vertices $\{0,1,\ldots,\frac{n}{2}-1\}$ and $\{\frac{n}{2},\ldots,n-1\}$ are  isomorphic. Now, we totally color the induced subgraphs using the total coloring of a complete graph of order $\frac{n}{2}$ (we give a similar total coloring to the edges between two vertices of the sub graph as in the complete graph). In order to color the edges  joining the two subgraphs, we observe that  the edges missed from the complete graph while coloring each subgraph is then later covered in the joining edges between the subgraphs. This is so possible because  if we have an edge between two vertices, say $g_i$ and $g_j$, we will not have an edge between $g_j$ and $g_i+\frac{n}{2}\pmod{\frac{n}{2}}$ or vice-versa; as this would imply that both $g_j-g_i\pmod{\frac{n}{2}}$ and $g_i+\frac{n}{2}-g_j=\frac{n}{2}-(g_j-g_i)\pmod{\frac{n}{2}}$ belong to the generating set, which is in contradiction to the assumption that we have none of $s_i+s_j\neq \frac{n}{2}$. Therefore, we could give the color to the edge between $g_i$ and $g_j$ which is missed from the complete graph to the edge between $g_i$ and $g_j+\frac{n}{2}\pmod{\frac{n}{2}}$ in $G$. Hence, we could color totally $G$ using the same colors as in the complete graph. Since we require $\frac{n}{2}+1$ colors to totally color a complete graph of order $\frac{n}{2}$, we are done.
\end{proof}
\begin{exmp}
Consider the circulant graph on $\mathbb{Z}_{24}$ with generating set $\{1, 3, 4, 5, 10, 14, 19,\\ 20, 21, 23\}$. We see that it satisfies the conditions of the theorem and hence can be totally colored with $13$ colors.
The total color matrix in this case can be written as in Table 4.
\begin{table}[h]
\begin{adjustbox}{width=\textwidth}
\begin{tabular}{|l|l|l|l|l|l|l|l|l|l|l|l|l|l|l|l|l|l|l|l|l|l|l|l|l|}\hline
   & 0  & 1  & 2  & 3  & 4  & 5  & 6  & 7  & 8  & 9  & 10 & 11 & 12 & 13 & 14 & 15 & 16 & 17 & 18 & 19 & 20 & 21 & 22 & 23 \\\hline
0  & 1  & 8  &   &  9 &  3 & 10 &   &  &   &  & 6  &  &    &    & 2  &    &    &    &    & 11 & 5  & 12 &    & 7  \\\hline
1  & 8  & 2  & 9  &    & 10 & 4  & 11 &    &    &    &    & 7  &    &    &    & 3  &    &    &    &    & 12 & 6  & 13 &    \\\hline
2  &    & 9  & 3  & 10 &    & 11 & 5  & 12 &    &    &    &    & 2  &    &    &    & 4  &    &    &    &    & 13 & 7  & 1  \\\hline
3  & 9  &    & 10 & 4  & 11 &    & 12 & 6  & 13 &    &    &    &    & 3  &    &    &    & 5  &    &    &    &    & 1  & 8  \\\hline
4  & 3  & 10 &    & 11 & 5  & 12 &    & 13 & 7  & 1  &    &    &    &    & 4  &    &    &    & 6  &    &    &    &    & 2  \\\hline
5  & 10 & 4  & 11 &    & 12 & 6  & 13 &    & 1  & 8  & 2  &    &    &    &    & 5  &    &    &    & 7  &    &    &    &    \\\hline
6  &    & 11 & 5  & 12 &    & 13 & 7  & 1  &    & 2  & 9  & 3  &    &    &    &    & 6  &    &    &    & 8  &    &    &    \\\hline
7  &    &    & 12 & 6  & 13 &    & 1  & 8  & 2  &    & 3  & 10 & 11 &    &    &    &    & 7  &    &    &    & 9  &    &    \\\hline
8  &    &    &    & 13 & 7  & 1  &    & 2  & 9  & 3  &    & 4  & 5  & 12 &    &    &    &    & 8  &    &    &    & 10 &    \\\hline
9  &    &    &    &    & 1  & 8  & 2  &    & 3  & 10 & 4  &    & 12 & 6  & 13 &    &    &    &    & 9  &    &    &    & 11 \\\hline
10 & 6  &    &    &    &    & 2  & 9  & 3  &    & 4  & 11 & 5  &    & 13 & 7  & 1  &    &    &    &    & 10 &    &    &    \\\hline
11 &    & 7  &    &    &    &    & 3  & 10 & 4  &    & 5  & 12 & 13 &    & 1  & 8  & 2  &    &    &    &    & 11 &    &    \\\hline
12 &    &    & 2  &    &    &    &    & 11 & 5  & 12 &    &    & 1  & 8  &    & 9  & 3  & 10 &    &    &    &    & 6  &    \\\hline
13 &    &    &    & 3  &    &    &    &    & 12 & 6  & 13 &    &    & 2  & 9  &    & 10 & 4  & 11 &    &    &    &    & 7  \\\hline
14 & 2  &    &    &    & 4  &    &    &    &    & 13 & 7  & 1  &    &    & 3  & 10 &    & 11 & 5  & 12 &    &    &    &    \\\hline
15 &    & 3  &    &    &    & 5  &    &    &    &    & 1  & 8  & 9  &    &    & 4  & 11 &    & 12 & 6  & 13 &    &    &    \\\hline
16 &    &    & 4  &    &    &    & 6  &    &    &    &    & 2  & 3  & 10 &    &    & 5  & 12 &    & 13 & 7  & 1  &    &    \\\hline
17 &    &    &    & 5  &    &    &    & 7  &    &    &    &    & 10 & 4  & 11 &    &    & 6  & 13 &    & 1  & 8  & 2  &    \\\hline
18 &    &    &    &    & 6  &    &    &    & 8  &    &    &    &    & 11 & 5  & 12 &    &    & 7  & 1  &    & 2  & 9  & 3  \\\hline
19 & 11 &    &    &    &    & 7  &    &    &    & 9  &    &    &    &    & 12 & 6  & 13 &    &    & 8  & 2  &    & 3  & 10 \\\hline
20 & 5  & 12 &    &    &    &    & 8  &    &    &    & 10 &    &    &    &    & 13 & 7  & 1  &    &    & 9  & 3  &    & 4  \\\hline
21 & 12 & 6  & 13 &    &    &    &    & 9  &    &    &    & 11 &    &    &    &    & 1  & 8  & 2  &    &    & 10 & 4  &    \\\hline
22 &    & 13 & 7  & 1  &    &    &    &    & 10 &    &    &    & 6  &    &    &    &    & 2  & 9  & 3  &    &    & 11 & 5  \\\hline
23 & 7  &    & 1  & 8  & 2  &    &    &    &    & 11 &    &    &    & 7  &    &    &    &    & 3  & 10 & 4  &    &    & 12\\\hline
\end{tabular}
\end{adjustbox}
\caption{Total Coloring of the circulant graph in Example 3.2.}
\label{table:4}
\end{table}

\end{exmp}

\begin{thm}
If a circulant graph $G$ of even order $n$ has a generating set $S$ such that the following properties hold:\\
(i) It does not have the involution element $\frac{n}{2}.$\\
(ii) It has a subset $M$ of $S$ such that $|M|=\frac{n}{2}-2$ and  $m_i+m_j\neq \frac{n}{2}, \ m_i,m_j\in M.$ \\
(iii) The complement of $M$ in $S$ generates the whole group.\\
Then, the graph $G$ satisfies $\chi''(G)\le\Delta(G)+3$.
\end{thm}
\begin{proof}
From the previous theorem, we see that the total chromatic number of the subgraph induced by the generating set $M$ is bounded above by $\frac{n}{2}+1$. Now, as the complement of $M$ in $S$ generates the whole group; From Lemma 1.1, the graph induced by the remaining edges form a Cayley graph and it is 1-factoriziable. Therefore we can color the remaining edges with $\Delta(G)-(\frac{n}{2}-2)=\Delta(G)-\frac{n}{2}+2$ colors. Thus, the number of colors required for the total coloring of $G$ is at most  $\frac{n}{2}+1+\Delta(G)-\frac{n}{2}+2=\Delta(G)+3$ colors.
\end{proof}
\begin{exmp}
Consider the circulant graph on $\mathbb{Z}_{24}$ with generating set

\noindent $\{1,2,3,4,5,7,10,14,17,19,20,21,23\}$.We see that this graph satisfies the conditions of the theorem. Here we have   $M=\{1,3,4,5,10,14,19,20,21,23\}$. From theorem 2.2, we see that the graph induced by the set $M$ requires $13$ colors. The complement of $M$ in $S$, $\{2,7,17,22\}$ has cardinality of $4$, whence the extra colors required are just $4$. Hence the graph can be totally colored with $17$ colors.
\end{exmp}
 
\begin{thm}
 Let $G$ be a circulant graph with $n=2m$ with $m$ odd  and $\Delta(G)> m$. If the generating sets $S$ and $S_1 \subset S$, $|S_1|=m-1$, have the following properties:\\
 (i).  $s_i\not\equiv s_j\pmod{\left\lceil\frac{n}{4}\right\rceil}$ where $s_i,s_j\in S_1$, $\frac{n}{2} \notin S_1 $. \\
 (ii). $S_1$   has a group generator in it.\\
 (iii). Complement of $S_1$ in $S$ generates the whole group.\\
 Then $\chi''_{=}(G)=\Delta(G)+1$ and $\chi_{\Sigma}(G)\le\Delta(G)+3$.
\end{thm}
\begin{proof}
We  first totally color  the subgraph induced by the generating subset $S_1$. For this, we first color the vertices with the canonical colors $1,2,3,\ldots \frac{n}{2}$ which implies that we have two vertices in each color class and the diagonal of the total color matrix is filled in the pattern $1-2-\ldots\frac{n}{2}-1-2\ldots , \frac{n}{2}$. This is possible as $\frac{n}{2}\notin S_1$. Now, to fill the subdiagonals, we use the same pattern as for the diagonals with the starting colors being different. The starting colors (entries in the first row) is chosen as follows:\\

For $s_i<\frac{n}{2}$, the entries corresponding to $s_i$ in the first row is chosen by using  the index of $s_i$ $(s_i \mod{\left\lceil\frac{n}{4}\right\rceil})$ in the canonical total coloring of a complete graph of order $\frac{n}{2}$.  As for the subdiagonals starting from $s_i>\frac{n}{2}$ , we use the colors used for inverses of $s_i\mod{\left\lceil\frac{n}{4}\right\rceil}$ in the canonical total coloring of the complete graph. \\

This ensures that there will be no clashes between the edge colors  (as $s_i\not\equiv s_j\mod{\left\lceil\frac{n}{4}\right\rceil}$). This also ensures that there is no clashes between the vertex and edge colors, because, if there are clashes between vertex and edge colors for $s_i>\frac{n}{2}$, there will be clashes at the colors between the vertices and subdiagonals starting at $s_i<\frac{n}{2}$, which is clearly not possible as $s_i\not\equiv s_j\pmod{\left\lceil\frac{n}{4}\right\rceil}$. Thus, we can give a type I total coloring for this subgraph using $m$ colors. Now, the  complement of $S_1$ in $S$ generates the whole group and by Lemma 1.1, the remaining edges of $G$ are one factoriziable. Therefore $G$ is type I.   \\

Since, in the above total coloring,  we have equitable coloring of vertices  and the graph is type I, each total color class having vertices has a perfect matching of the remaining graph (the graph with the two vertices removed). The total color classes having only edges  have just one  element less than the total color classes with vertices and as such, the graph has an equitable total coloring using $\Delta(G)+1$ colors as desired. Thus, $\chi''_{=}(G)=\chi''(G)=\Delta(G)+1.$ \\

In order to prove that $\chi_{\Sigma}(G)\le\Delta(G)+3$, we first give a $\Delta(G)+1$ total coloring of $G$ as above and then take a  Hamiltonian cycle, (since a group generator is there in the generating set of the graph by property ii) which has two rainbow perfect matchings. We give two extra colors to its edges. This ensures that the sum of colors of neighbors is different. Hence, $\chi_{\Sigma}(G)\le(\Delta(G)+1)+2=\Delta(G)+3$.
\end{proof}

\begin{exmp}
We take the graph $G$ the circulant graph on $\mathbb{Z}_{18}$ with generating set $S=\{1,2,4,6,7,8,10,11,12,14,16,17\}$. Here, $S_1$ is $\{1,2,4,6,12,14,16,17\}$. We first give a $9$-total coloring of the subgraph induced by $S_1$. To give a total $9$-coloring, we use the total color matrix of the canonical coloring of a complete graph of order $9$. This gives us the starting colors for the subdiagonals starting at $1,2,4,6$ to be $6,2,3,4$. correspondingly the subdiagonals starting at $12,14,16,17$ to be $7,8,9,5$ .The part of graph induced by $\{7,8,12,13\}$, which is also a circulant graph, can be given a $4$ coloring ( from Corollary 2.3.1 of \cite {STO}). This gives us a $13$ total coloring of the graph which is also equitable. Now, we take the rainbow Hamiltonian cycle  induced by the set $\{1,19\}$  and give it a $2$- coloring using two extra colors (replacing the original coloring), which is possible as the cycle induced is Hamiltonian (and hence even). This gives us a neighborhood sum distinguishing total coloring using $15$ colors as the colors replaced in each row and/or column happens to be having a distinct sum owing to the rainbow nature of the prior coloring. The total color matrix of the subgraph corresponding to $S_1$ is given in Table 5.\\
The final total color matrix can be obtained by using the Lemma 1.1 and $4$ extra colors to give us  an equitable $13$- coloring of the final graph. \\
\begin{table}[h]
\begin{tabular}{|l|l|l|l|l|l|l|l|l|l|l|l|l|l|l|l|l|l|l|}
\hline
   & 0  & 1  & 2  & 3  & 4  & 5  & 6  & 7  & 8  & 9  & 10 & 11 & 12 & 13 & 14 & 15 & 16 & 17 \\\hline
0  & 1  & 6  & 2  &    & 3  &    & 4  &  &  &    &   &  &7  &    & 8  &    & 9  & 5  \\\hline
1  & 6  & 2  & 7  & 3  &    & 4  &    & 5  &  &  &    &   &  &8  &    & 9  &    & 1  \\\hline
2  & 2  & 7  & 3  & 8  & 4  &    & 5  &    & 6  &  &  &    &   &  &9  &    & 1  &    \\\hline
3  &    & 3  & 8  & 4  & 9  & 5  &    & 6  &    & 7  &  &  &    &   &  &1  &    & 2  \\\hline
4  & 3  &    & 4  & 9  & 5  & 1  & 6  &    & 7  &    & 8  &  &  &    &   &  &2  &    \\\hline
5  &    & 4  &    & 5  & 1  & 6  & 2  & 7  &    & 8  &    & 9  &  &  &    &   &  & 3 \\\hline
6  & 4  &    & 5  &    & 6  & 2  & 7  & 3  & 8  &    & 9  &    & 1  &  &  &    &   &  \\\hline
7  &  & 5  &    & 6  &    & 7  & 3  & 8  & 4  & 9  &    & 1  &    & 2  &  &  &    &   \\\hline
8  &  &  & 6  &    & 7  &    & 8  & 4  & 9  & 5  & 1  &    & 2  &    & 3  &  &  &    \\\hline
9  &    &  &  & 7  &    & 8  &    & 9  & 5  & 1  & 6  & 2  &    & 3  &    & 4  &  &  \\\hline
10 &   &    &  &  & 8  &    & 9  &    & 1  & 6  & 2  & 7  & 3  &    & 4  &    & 5  &  \\\hline
11 &    &   &    &  &  & 9  &    & 1  &    & 2  & 7  & 3  & 8  & 4  &    & 5  &    & 6  \\\hline
12 & 7 &    &   &    &  &  & 1  &    & 2  &    & 3  & 8  & 4  & 9  & 5  &    & 6  &    \\\hline
13 &  & 8 &    &   &    &  &  & 2  &    & 3  &    & 4  & 9  & 5  & 1  & 6  &    & 7  \\\hline
14 & 8  &  &9  &    &   &    &  &  & 3  &    & 4  &    & 5  & 1  & 6  & 2  & 7  &    \\\hline
15 &    & 9  &  & 1 &    &   &    &  &  & 4  &    & 5  &    & 6  & 2  & 7  & 3  & 8  \\\hline
16 & 9  &    & 1  &  &2  &    &   &    &  &  & 5  &    & 6  &    & 7  & 3  & 8  & 4  \\\hline
17 & 5 & 1  &    & 2  &  & 3 &    &   &    &  &  & 6  &    & 7  &    & 8  & 4  & 9 \\\hline
\end{tabular}
\caption{Partial equitable total Coloring of the Circulant graph in Example 3.4.}
\label{table:5}
\end{table}

\newpage
To give  a neighborhood sum distinguishing total coloring using $15$ colors, we modify the prior total color matrix  by replacing the colors of the Rainbow Hamiltonian cycle induced by the set $\{1,19\}$ as in Table 6.
\begin{table}[h]
\begin{adjustbox}{width=\textwidth}
\begin{tabular}{|l|l|l|l|l|l|l|l|l|l|l|l|l|l|l|l|l|l|l|}
\hline
   & 0  & 1  & 2  & 3  & 4  & 5  & 6  & 7  & 8  & 9  & 10 & 11 & 12 & 13 & 14 & 15 & 16 & 17 \\\hline
0  & 1  & 14 & 2  &    & 3  &    & 4  &  &  &    & 6  &  &  &    & 8  &    & 9  & 15 \\\hline
1  & 14 & 2  & 15 & 3  &    & 4  &    & 5  &  &  &    & 7  &  &  &    & 9  &    & 1  \\\hline
2  & 2  & 15 & 3  & 14 & 4  &    & 5  &    & 6  &  &  &    & 8  &  &  &    & 1  &    \\\hline
3  &    & 3  & 14 & 4  & 15 & 5  &    & 6  &    & 7  &  &  &    & 9  &  &  &    & 2  \\\hline
4  & 3  &    & 4  & 15 & 5  & 14 & 6  &    & 7  &    & 8  &  &  &    & 1  &  &  &    \\\hline
5  &    & 4  &    & 5  & 14 & 6  & 15 & 7  &    & 8  &    & 9  &  &  &    & 2  &  &  \\\hline
6  & 4  &    & 5  &    & 6  & 15 & 7  & 14 & 8  &    & 9  &    & 1  &  &  &    & 3  &  \\\hline
7  &  & 5  &    & 6  &    & 7  & 14 & 8  & 15 & 9  &    & 1  &    & 2  & 1 &  &    & 4  \\\hline
8  &  &  & 6  &    & 7  &    & 8  & 15 & 9  & 14 & 1  &    & 2  &    & 3  &  &  &    \\\hline
9  &    &  &  & 7  &    & 8  &    & 9  & 14 & 1  & 15 & 2  &    & 3  &    & 4  &  &  \\\hline
10 & 6  &    &  &  & 8  &    & 9  &    & 1  & 15 & 2  &14 & 3  &    & 4  &    & 5  &  \\\hline
11 &    & 7  &    &  &  & 9  &    & 1  &    & 2  & 14 & 3  & 15 & 4  &    & 5  &    & 6  \\\hline
 &  &    & 8  &    &  &  & 1  &    & 2  &    & 3  & 15 & 4  & 14 & 5  &    & 6  &    \\\hline
 &  &  &    & 9  &    &  &  & 2  &    & 3  &    & 4  & 14 & 5  & 15 & 6  &    & 7  \\\hline
14 & 8  &  &  &    & 1  &    &  &  & 3  &    & 4  &    & 5  & 15 & 6  & 14 & 7  &    \\\hline
15 &    & 9  &  &  &    & 2  &    &  &  & 4  &    & 5  &    & 6  & 14 & 7  & 15 & 8  \\\hline
16 & 9  &    & 1  &  &  &    & 3  &    &  &  & 5  &    & 6  &    & 7  & 15 & 8  & 14 \\\hline
17 & 15 & 1  &    & 2  &  &  &    & 4  &    &  &  & 6  &    & 7  &    & 8  & 14 & 9\\ \hline
\end{tabular}
\end{adjustbox}
\caption{Partial total coloring of the circulant graph in Example 3.4 that is neighborhood sum distinguishing.}
\label{table:6}
\end{table}
The final total color matrix can be obtained by using Lemma 1.1 and $4$ extra colors to give us  a total coloring with $\chi''_{\Sigma}(G)=15$.
\end{exmp}
\newpage
\noindent \textbf{Acknowledgements:} \\
The first author was supported by CSIR Fellowship (File No. 09/942(0020)/2020-EMR-I) .\\
The second and third authors were supported by SERB, India (No.SERB: EMR/2017/001869).


\begin{thebibliography}{99}
\bibitem{BEH}Behzad, Mehdi. Graphs and their chromatic numbers. Michigan State University, 1965.

\bibitem{CAM}Campos, C. N., and de Mello, Célia Picinin. ``A result on the total colouring of powers of cycles." Discrete Applied Mathematics 155.5 (2007): 585-597.

\bibitem{FU}Fu, H-L. ``Some results on equalized total coloring." Congressus Numerantium (1994): 111-120.

\bibitem{GNS}  Geetha,J.,   Narayanan, N. and  Somasundaram, K. ``Total Colourings-A survey." arXiv preprint arXiv:1812.05833 (2018).

\bibitem{Gee} Geetha, J., K. Somasundaram, and Hung-Lin Fu. "Total colorings of circulant graphs." Discrete Mathematics, Algorithms and Applications (2020): 2150050.


\bibitem{PIL} Pilśniak, Monika, and  Woźniak, Mariusz. ``On the total-neighbor-distinguishing index by sums." Graphs and Combinatorics 31.3 (2015): 771-782.

\bibitem{STO} Stong, Richard A. ``On 1-factorizability of Cayley graphs." Journal of Combinatorial Theory, Series B 39.3 (1985): 298-307.
\bibitem{VGV} Vizing, Vadim G. ``Some unsolved problems in graph theory." Russian Mathematical Surveys 23.6 (1968): 125.

\bibitem{WANG}Wang, Wei-Fan. ``Equitable total coloring of graphs with maximum degree 3." Graphs and Combinatorics 18.3 (2002): 677-685.

\end{thebibliography}
\end{document}